\documentclass[11pt]{article}
\usepackage{amsmath,amssymb}

\newtheorem{propo}{{\bf Proposition}}[section]
\newtheorem{coro}[propo]{{\bf Corollary}}
\newtheorem{lemma}[propo]{{\bf Lemma}} \newtheorem{theor}[propo]{{\bf
Theorem}} 

\newenvironment{proof}{{\bf Proof.}}{$\Box$}

\def\N{{\mathbb N}}

\begin{document}

\vspace*{1.0in}

\begin{center} LEIBNIZ $A$-ALGEBRAS
\end{center}
\bigskip

\begin{center} DAVID A. TOWERS 
\end{center}
\bigskip

\begin{center} Department of Mathematics and Statistics

Lancaster University

Lancaster LA1 4YF

England

d.towers@lancaster.ac.uk 
\end{center}
\bigskip

\begin{abstract}
A finite-dimensional Lie algebra is called an A-algebra if all of its nilpotent subalgebras are abelian. These arise in the study of constant Yang-Mills potentials and have also been particularly important in relation to the problem of describing residually finite varieties. They have been studied by several authors, including Bakhturin, Dallmer, Drensky, Sheina, Premet, Semenov, Towers and Varea. In this paper we establish generalisations of many of these results to Leibniz algebras.
\par 
\noindent {\em Mathematics Subject Classification 2000}: 17B05, 17B20, 17B30, 17B50.
\par
\noindent {\em Key Words and Phrases}: Lie algebras, Leibniz algebras, $A$-algebras, Frattini ideal, solvable, nilpotent, completely solvable, metabelian, monolithic, cyclic Leibniz algebras. 
\end{abstract}

\section{Introduction}
\medskip

An algebra $L$ over a field $F$ is called a {\em Leibniz algebra} if, for every $x,y,z \in L$, we have
\[  [x,[y,z]]=[[x,y],z]-[[x,z],y]
\]
In other words the right multiplication operator $R_x : L \rightarrow L : y\mapsto [y,x]$ is a derivation of $L$. As a result such algebras are sometimes called {\it right} Leibniz algebras, and there is a corresponding notion of {\it left} Leibniz algebra. Every Lie algebra is a Leibniz algebra and every Leibniz algebra satisfying $[x,x]=0$ for every element is a Lie algebra. They were introduced in 1965 by Bloh (\cite{bloh}) who called them $D$-algebras, though they attracted more widespread interest, and acquired their current name, through work by Loday and Pirashvili ({\cite{loday1}, \cite{loday2}).
\par

The {\it Leibniz kernel} is the set Leib$(L)=$ span$\{x^2:x\in L\}$. Then Leib$(L)$ is the smallest ideal of $L$ such that $L/$Leib$(L)$ is a Lie algebra.
Also $[L,$Leib$(L)]=0$.
\par

We define the following series:
\[ L^1=L,L^{k+1}=[L^k,L] \hbox{ and } L^{(0)}=L,L^{(k+1)}=[L^{(k)},L^{(k)}] \hbox{ for all } k=2,3, \ldots
\]
Then $L$ is {\em nilpotent of class n} (resp. {\em solvable of derived length n}) if $L^{n+1}=0$ but $L^n\neq 0$ (resp.$ L^{(n)}=0$ but $L^{(n-1)}\neq 0$) for some $n \in \N$. It is straightforward to check that $L$ is nilpotent of class n precisely when every product of $n+1$ elements of $L$ is zero, but some product of $n$ elements is non-zero.The {\em nilradical}, $N(L)$, (resp. {\em radical}, $R(L)$) is the largest nilpotent (resp. solvable) ideal of $L$.
\par

A Lie algebra $L$ is called an $A$-algebra if all of its nilpotent subalgebras are abelian. This is analogous to the concept of an $A$-group: a finite group with the property that all of its Sylow subgroups are abelian.  They have been studied and used by a number of authors, including Bakhturin and Semenov \cite{bs}, Dallmer \cite{dall}, Drensky \cite{dren}, Sheina \cite{sh1} and \cite{sh2}, Premet and Semenov \cite{ps}, Semenov \cite{sem} and Towers and Varea \cite{tv1}, \cite{tv2}.
They arise in the study of constant Yang-Mills potentials and have also been particularly important in relation to the problem of describing residually finite varieties (see \cite{bs}, \cite{sh1}, \cite{sh2}, \cite{sem} and \cite{ps}). 
\par

It would seem to be worthwhile examining this same concept for Leibniz algebras, both because there has been much interest in seeing which properties of Lie algebras generalise to Leibniz algebras, but also because Leibniz algebras can be used to define consistent generalisations of Yang-Mills functionals. 
\par

In section two we consider the non-solvable case. Here we collect together the preliminary results that we need, including the fact that for Leibniz $A$-algebras the derived series coincides with the lower nilpotent series. The main result is an analogue of the structure theorem of Premet and Semenov (\cite{ps}).
\par
Section three contains the basic structure theorems for solvable Leibniz $A$-algebras. First they split over each term in their derived series. This leads to a decomposition of $L$ as $L = A_{n} \dot{+} A_{n-1} \dot{+} \ldots \dot{+} A_0$ where $A_i$ is an abelian subalgebra of $L$ and $L^{(i)} = A_{n} \dot{+} A_{n-1} \dot{+} \ldots \dot{+} A_{i}$ for each $0 \leq i \leq n$. It is shown that the ideals of $L$ relate nicely to this decomposition: if $K$ is an ideal of $L$ then $K = (K \cap A_n) \dot{+} (K \cap A_{n-1}) \dot{+} \ldots \dot{+} (K \cap A_0)$; moreover, if $N$ is the nilradical of $L$, $Z(L^{(i)}) = N \cap A_i$. We also see that the result in Theorem \ref{t:a} (iii)(a) holds when $L$ is solvable without any restrictions on the underlying field. 
\par
The fourth section looks at Leibniz $A$-algebras in which $L^2$ is nilpotent. These are metabelian and so the results of section three simplify. In addition we can locate the position of the maximal nilpotent subalgebras: if $U$ is a maximal nilpotent subalgebra of $L$ then $U = (U \cap L^2) \oplus (U \cap C)$ where $C$ is a Cartan subalgebra of $L$. 
\par
Section five is devoted to Leibniz $A$-algebras having a unique minimal ideal $W$. Again some of the results of sections three and four simplify. In particular, $N = Z_L(W)$, and if $L$ is strongly solvable the maximal nilpotent subalgebras of $L$ are $L^2$ and the Cartan subalgebras of $L$ (that is, the subalgebras that are complementary to $L^2$.) We also give necessary and sufficient conditions for a Leibniz algebra with a unique minimal ideal to be a strongly solvable $A$-algebra.  
\par
In section six we illustrate some of the previous results by examining the subclass of cyclic Leibniz algebras.
\par 
The final section is devoted to generalising a result of Drensky (\cite{dren}). This shows that a solvable Leibniz algebra over an algebraically closed field has derived length at most three.
\par
Throughout $L$ will denote a finite-dimensional algebra over a field $F$. Algebra direct sums will be denoted by $\oplus$, whereas vector space direct sums will be denoted by $\dot{+}$. The {\it centre} of $L$ is $Z(L)=\{z\in L \mid [z,x]=[x,z]=0$ for all $x\in L\}$. If $U$ is a subalgebra of $L$, the {\it centraliser} of $U$ in $L$ is $C_L(U)=\{x\in L \mid [x,U]=[U,x]=0\}$. We say that $L$ is {\it monolithic} with {\it monolith $W$} if $W$ is the unique minimal ideal of $L$. The Frattini ideal of $L$, $\phi(L)$, is the largest ideal of $L$ contained in all maximal subalgebras of $L$; we call $L$ {\it $\phi$-free}
if $\phi(L) = 0$.

\section{The non-solvable case}
First we note that the class of Leibniz $A$-algebras is closed with respect to subalgebras, factor algebras and direct sums. Also that there is always a unique maximal abelian ideal, and it is the nilradical.

\begin{lemma}\label{l:lemm1}
Let $L$ be a Lie $A$-algebra and let $N$ be its nilradical. Then
\begin{itemize}
\item[(i)] $N$ is the unique maximal abelian ideal of $L$;
\item[(ii)] if $B$ and $C$ are abelian ideals of $L$, we have $[B, C] = 0$.
\end{itemize}
\end{lemma}
\begin{proof} (i) Clearly $N$ is abelian and contains every abelian ideal of $L$.
\par

(ii) Simply note that $B,C\subseteq N$.
\end{proof}

\begin{lemma}\label{l:fac} If $L$ is a Leibniz $A$-algebra over any field and $B$ is an ideal of $L$, then $L/B$ is a Leibniz $A$-algebra.
\end{lemma}
\begin{proof} Let $U$ be a subalgebra of $L$ such that $U/B$ is nilpotent. If $B\subseteq \phi(U)$ the $U$ is nilpotent (see \cite{barnes}) and hence abelian. 
\par

So suppose that $B \not \subseteq \phi(U)$. Then there is a maximal subalgebra $M$ of $U$ such that $U=B+M$. Choose $C$ to be a subalgebra of $L$ which is minimal with respect to $U=B+C$. Then $B\cap C\subseteq \phi(C)$ and $U/B \cong C/B\cap C$. It follows, by \cite{barnes} again, that $C$ is nilpotent and hence abelian.
\par

So, in either case, $U/B$ is abelian and $L/B$ is an $A$-algebra.  
\end{proof}

\begin{lemma}\label{l:lemm2} Let $B$, $C$ be ideals of the Leibniz algebra $L$.
\begin{itemize}
\item[(i)] If $L/B$, $L/C$ are $A$-algebras, then $L/(B \cap C)$ is an $A$-algebra.
\item[(ii)] If $L = B \oplus C$, where $B, C$ are $A$-algebras, then $L$ is an $A$-algebra.
\end{itemize}
\end{lemma}
\begin{proof} (i) Let $U/(B \cap C)$ be a nilpotent subalgebra of $L/(B \cap C)$. Then $(U + B)/B$ is a nilpotent subalgebra of $L/B$, which is an $A$-algebra. It follows that $U^2 \subseteq B$. Similarly, $U^2 \subseteq C$, whence the result.
\par
(ii) This follows from (i).
\end{proof}
\medskip

We define the {\em nilpotent residual}, $\gamma_{\infty}(L)$, of $L$ be the smallest ideal of $L$ such that $L/\gamma_{\infty}(L)$ is nilpotent. Clearly this is the intersection of the terms of the lower central series for $L$. Then the {\em lower nilpotent series} for $L$ is the sequence of ideals $N_i(L)$ of $L$ defined by $N_0(L) = L$, $N_{i+1}(L) = \gamma_{\infty}(N_i(L))$ for $i \geq 0$. 
\par
For Leibniz $A$-algebras we have the following result.

\begin{lemma}\label{l:series}
Let $L$ be a Leibniz $A$-algebra. Then the lower nilpotent series coincides with the derived series.
\end{lemma}
\medskip
{\it Proof.} Since $L/L^{(1)}$ is nilpotent we have $N_1(L) \subseteq L^{(1)}$. Also $L/N_1(L)$ is nilpotent and hence abelian, by Lemma \ref{l:lemm1} (ii), so $L^{(1)} \subseteq N_1(L)$. Repetition of this argument gives $N_i(L) = L^{(i)}$ for each $i \geq 0$.
\bigskip

If $F$ has characteristic zero, then every solvable Leibniz $A$-algebra over $F$ is metabelian, since $L^2$ is nilpotent. This is not the case, however, when $F$ is any field of characteristic $p > 0$ (see \cite[Example 2.1]{tow}).
\par

A main problem encountered when trying to generalise results about Lie algebras to the case of Leibniz algebras is the lack of anti-symmetry, so that one-sided ideals exist in a Leibniz algebra. The following lemma is used several times in this paper to overcome this difficulty.

\begin{lemma}\label{l:leftideal} Let $A$ be an abelian ideal of a Leibniz algebra $L$ and suppose that $x^2\in A$. Then $L_x^n(A)\subseteq R_x^{n-1}(A)$ for all $n\geq 1$.
\end{lemma}
\begin{proof} Clearly $[x,A]\subseteq A$ so the result holds for $n=1$. Suppose that it holds for $n\leq k$ where $k\geq 1$. Then
\begin{align}
 L_x^{k+1}(A)& = [x,[x,L_x^{k-1}(A)]]\subseteq [x^2,L_x^{k-1}(A)]+[[x,L_x^{k-1}(A)],x] \nonumber \\
&= [L_x^k(A),x]\subseteq R_x^k(A). \nonumber
\end{align} The result follows by induction.
\end{proof}
\medskip

Finally in this section we generalise a structure theorem of Premet and Semenov (see \cite{ps}) to Leibniz algebras. We will need the following easy lemma.

\begin{lemma}\label{l:3dsplit} Let $L$ be a Leibniz algebra over a field of characteristic different from $2$ such that $L/Z(L)$ is a simple three-dimensional Lie algebra. Then $L= L^2\dot{+} Z(L)$.
\end{lemma}
\begin{proof} By \cite[page 13]{jac}, $L/Z(L)$ has a basis $e_1+Z(L),e_2+Z(L),e_3+Z(L)$ with products $[e_2,e_3]+Z(L)=e_1+Z(L)$,$[e_3,e_1]+Z(L)=\alpha e_2+Z(L)$, $[e_1,e_2]+Z(L)=\beta e_3+Z(L)$ for some $\alpha,\beta \in F\setminus{0}$. Then it is easy to see that the subspace $S$ spanned by $[e_1,e_2]$,$[e_3,e_1]$,$[e_3,e_2]$ is a three dimensional simple subalgebra of $L$. It follows that $Z(L)\cap S=0$ and $S=L^2$. Hence $L=L^2\dot{+}Z(L)$.
\end{proof}
\medskip

If $K$ is an extension field of $F$, denote $K\otimes_F L$ by $L_K$.

\begin{theor}\label{t:a}
Let $L$ be a Leibniz $A$-algebra over a field $F$. If $F$ has characteristic $\neq 2, 3$ and cohomological dimension $\leq 1$ (this means that the Brauer group of any algebraic extension of the underlying field is trivial), then
\begin{itemize}
\item[(i)] $L^2 \cap Z(L) = 0$; and
\item[(ii)] $L$ has a Levi decomposition and every Levi subalgebra is representable as a direct sum of simple ideals, each one of which splits over some finite extension of the ground field into a direct sum of ideals isomorphic to $sl(2)$.
\end{itemize}
\end{theor}
\begin{proof}
\begin{itemize}
\item[(i)] Let $L$ be a minimal counter-example, so there is a non-zero element $x\in Z(L)\cap L^2$. Clearly Leib$(L) \neq 0$ by \cite[Proposition 2]{ps}.  Let $A$ be a subspace complementary to $Fx$ in $Z(L)$, so $Z(L)=Fx\dot{+}A$. Then 
$$x+A\in \frac{Z(L)}{A}\cap \frac{L^2+A}{A}\subseteq Z\left(\frac{L}{A}\right)\cap \left(\frac{L}{A}\right)^2,$$ 
so we have that $A=0$ and $\dim Z(L)=1$. If $B$ is a non-trivial ideal of $L$ we have $Z(L)\subseteq B$, since otherwise $L/B$ would be a counter-example of smaller dimension. It follows that $L$ is monolithic with monolith $Z(L)$. Let $M$ be a maximal ideal of $L$. Then $M^2\cap Z(M)=0$ and so $Z(L)\not\subseteq M^2$, whence $M^2=0$. But now either $L$ is nilpotent or there is a unique maximal ideal which is abelian and is the radical. If $L$ is nilpotent, it is abelian, which yields a contradiction.
\par

So suppose that $L$ has a unique maximal  ideal $M$ which is abelian and is the radical. Then $L/M={\mathcal L}$ is simple. It follows from \cite[Corollary 1 and Lemma 2]{ps} that ${\mathcal L}$ is a Lie $p$-algebra Moreover, our assumption on the field $F$ implies that${\mathcal L}$ has a non-zero nilpotent element (see \cite{p1} and \cite{p2}). Hence there exists an element $u \in L\setminus M$ such that $R_u^{p^m}(L)\subseteq M$. Let $\overline{u}$ be the image of $u$ under the canonical homomorphism from $L$ to $\mathcal{L}$. The element $\overline{u}^{p^m}$ lies in the centre of the universal enveloping algebra $U({\mathcal L})$, and so in any indecomposable $L$-module $W$ the set $\lambda_1(W), \ldots, \lambda_r(W)$ of eigenvalues of $\overline{u}^{p^m}$ consists of elements of $\overline{F}$ that are conjugate under the Galois group $Gal(\overline{F}/F)$. The right module $M$ is indecomposable and contains $Z(L)$, and so $\lambda_k(M)=0$ for some $1 \leq k \leq r$. It follows that $u$ acts nilpotently on the right in $L$. But now $u^2\in$ Leib$(L)\subseteq M$, so, using Lemma \ref{l:leftideal}, $Fu+M$ is a nilpotent subalgebra of $L$ and thus abelian. This yields that $u\in C_L(M)$, and so $C_L(M)=L$ and $M=Z(L)$.
\par

Now there is a finite extension $K$ of $F$ over which $(L/Z(L))_K$ splits as a direct sum of ideals $S_1/Z(L) \oplus \ldots \oplus S_n/Z(L)$ isomorphic to $sl(2)$, by \cite[Proposition 2(ii)]{ps} again. Let $\theta: L \rightarrow L/Z(L)$ be the canonical homomorphism with $\ker(\theta)=Z(L)$ and let $\theta_K: L_K \rightarrow (L/Z(L))_K$ be the natural extension of $\theta$ to the corresponding algebras over the extension field. Then $\theta_K$ is a surjective homomorphism with $\ker(\theta_K)=(\ker(\theta))_K$ (see, for example, \cite{jac}), so $(L/Z(L))_K\cong L_K/Z(L)_K$. Using Lemma \ref{l:3dsplit} we thus see that $L_K=(L_K)^2 \dot{+}Z(L)_K$. But now $L=L^2\dot{+}Z(L)$, a contradiction from which the result follows.

\item[(ii)] We have that $L/$Leib$(L)=S/$Leib$(L)\dot{+}R/$Leib$(L)$ where $R$ is the radical of $L$ and there is a finite extension $K$ of $F$ over which $S/$Leib$(L)$ splits as a direct sum of ideals $S_1/$Leib$(L) \oplus \ldots \oplus S_n/$Leib$(L)$ isomorphic to $sl(2)$, by \cite[Proposition 2(ii)]{ps}. Arguing as in the final two paragraphs of (i) we have that $L_K=S_K^2\dot{+}R_K$, from which $L=S^2\dot{+}R$ giving the claimed result.
\end{itemize}
\end{proof}

\section{Decomposition results for Solvable Leibniz $A$-algebras}
Here we have the basic structure theorems for solvable Leibniz $A$-algebras. First we see that such an algebra splits over the terms in its derived series.

\begin{lemma}\label{l:nilrad}
Let $L$ be any solvable Leibniz algebra with nilradical $N$. Then $Z_L(N) \subseteq N$ 
\end{lemma}
\begin{proof} Suppose that $Z_L(N) \not \subseteq N$. Then there is a non-trivial abelian ideal $A/(N \cap Z_L(N)$ of $L/(N \cap Z_L(N)$ inside $Z_L(N)/(N \cap Z_L(N)$. But now $A^3 \subseteq [N, A] = 0$, so $A$ is a nilpotent ideal of $L$. It follows that $A \subseteq N \cap Z_L(N)$, a contradiction.
\end{proof}

\begin{theor}\label{t:split}
Let $L$ be a solvable Leibniz $A$-algebra. Then $L$ splits over each term in its derived series. Moreover, the Cartan subalgebras of $L^{(i)}/L^{(i+2)}$ are precisely the subalgebras that are complementary to $L^{(i+1)}/L^{(i+2)}$ for $i \geq 0$.
\end{theor}
\begin{proof} Suppose that $L^{(n+1)} = 0$ but $L^{(n)} \neq 0$. First we show that $L$ splits over $L^{(n)}$. Clearly we can assume that $n \geq 1$. Let $C$ be a Cartan subalgebra of $L^{(n-1)}$ (this exists in any solvable Leibniz algebra: the proof is the same as that for Lie algebras in \cite[Corollary 4.4.1.2]{wint}) and let $L = L_0 \dot{+} L_1$ be the Fitting decomposition of $L$ relative to $R_C$. Then $L_1 = \cap_{k=1}^{\infty} R_C^k(L) \subseteq L^{(n)}$, and so $L_1$ is an abelian right ideal of $L$. Also $L^{(n-1)} = L_1 \dot{+} L_0 \cap L^{(n-1)}$ and  $L_0 \cap L^{(n-1)} = (L^{(n-1)})_0 = C$, which is abelian. 
\par

Now 
\[ [C,[L,C]]\subseteq [[C,L],C] +[C^2,L]\subseteq [L,C]
\] Suppose that $[C,R_C^k(L)]\subseteq R_C^k(L)$ for $k\geq 1$. Then
\[ [C,R_C^{k+1}(L)] = [C,[R_C^k(L),C]]\subseteq [[C,R_C^k(L)],C]+[C^2,R_C^k(L)]\subseteq R_C^{k+1}(L).
\]
It follows that $[C,L_1]\subseteq L_1$ and thus that $L_1$ is an ideal of $L^{(n-1)}$. But $L^{(n-1)}/L_1$ is abelian, whence $L^{(n)} \subseteq L_1$ and $L = L_0 \dot{+} L^{(n)}$.
\par
So we have that $L = L^{(n)} \dot{+} B$ where $B = L_0$ is a subalgebra of $L$. Clearly $B^{(n)} = 0$, so, by the above argument, $B$ splits over $B^{(n-1)}$, say $B = B^{(n-1)} \dot{+} D$. But then $L = L^{(n)} \dot{+} (B^{(n-1)} \dot{+} D) = L^{(n-1)} \dot{+} D$. Continuing in this way gives the desired result.
\end{proof}
\medskip

This gives us the following fundamental decomposition result.

\begin{coro}\label{c:decomp}
Let $L$ be a solvable Lie $A$-algebra of derived length $n+1$. Then
\begin{itemize}
\item[(i)] $L = A_{n} \dot{+} A_{n-1} \dot{+} \ldots \dot{+} A_0$ where $A_i$ is an abelian subalgebra of $L$ for each $0 \leq i \leq n$; and
\item[(ii)] $L^{(i)} = A_{n} \dot{+} A_{n-1} \dot{+} \ldots \dot{+} A_{i}$ for each $0 \leq i \leq n$
\end{itemize}
\end{coro}
\begin{proof} (i) By Theorem \ref{t:split} there is a subalgebra $B_n$ of $L$ such that $L = L^{(n)} \dot{+} B_n$. Put $A_n = L^{(n)}$. Similarly $B_n = A_{n-1} \dot{+} B_{n-1}$ where $A_{n-1} = (B_n)^{(n-1)}$. Continuing in this way we get the claimed result. Note, in particular, that it is apparent from the construction that $A_k \cap (A_{k-1} + ... + A_0) = 0$ for each $1 \leq k \leq n$, and that it is easy to see from this that the sum is a vector space direct sum.
\par
(ii) We have that $L^{(n)} = A_n$. Suppose that $L^{(k)} = A_n \dot{+} \ldots \dot{+} A_k$ for some $1 \leq k \leq n$. Then $L = L^{(k)} \dot{+} B_k$ and $A_{k-1} = B_k^{(k-1)}$ by the construction in (i). But now $L^{(k-1)} \subseteq L^{(k)} + B_k^{(k-1)} \subseteq L^{(k-1)}$, whence $L^{(k-1)} = A_n \dot{+} \ldots \dot{+} A_{k-1}$ and the result follows by induction.
\end{proof}
\medskip

Now we show that the result in Theorem \ref{t:a} (iii)(a) holds when $L$ is solvable without any restrictions on the underlying field. 

\begin{theor}\label{t:int}
Let $L$ be a solvable Leibniz $A$-algebra. Then $Z(L) \cap L^2 = 0$.
\end{theor}
\begin{proof} Let $L$ be a minimal counter-example and let $z \in Z(L) \cap L^2$. Put $Z(L) = U \dot{+} Fz$. Then $U$ is an ideal of $L$ and 
$$ U \neq z + U \in (Z(L) \cap L^2 + U)/U \subseteq Z(L/U) \cap (L/U)^2. $$
The minimality of $L$ implies that $U = 0$, so $Z(L) = Fz$. But now if $K$ is an ideal of $L$ which does not contain $Z(L)$, then $K \neq z + K \in Z(L/K) \cap (L/K)^2$ similarly, contradicting the minimality of $L$. It follows that $L$ is monolithic with monolith $Z(L)$.
\par
Now let $M$ be a maximal ideal of $L$. Then $Z(M) \cap M^2 = 0$ by the minimality of $L$, so $Z(L) \not \subseteq M^2$, whence $M^2 = 0$. It follows that $L = M \dot{+} Fx$ for some $x \in L$ and $M$ is abelian. Let $L = L_0 \dot{+} L_1$ be the Fitting decomposition of $L$ relative to $R_x$. Then $L_1 = \cap_{i=1}^{\infty} R_x^i(L) \subseteq M$, and $[L_1,L_0] \subseteq L_1$, so $L_1$ is a right ideal of $L$. 
\par

Now 
\[ [x,[L,x]]\subseteq [[x,L],x]+[x^2,L]\subseteq [L,x]+[x^2,M+Fx] \subseteq [L,x]
\]
since $x^2\in I\subseteq M$, so $[x^2,M]=0$. Suppose that $[x,R_x^k(L)] \subseteq R_x^k(L)$. Then
\[ [x,R_x^{k+1}(L)]=[x,[R_x^k(L),x]]\subseteq [[x,R_x^k(L)],x]+[x^2,R_x^k(L)]\subseteq R_x^{k+1}(L),
\]
since $R_x^k(L)\subseteq [L,x]=[M+Fx,x]\subseteq M$, whence $[x^2,R_x^k(L)]=0$. It follows that $[L,L_1]=[x,L_1]\subseteq L_1$ and $L_1$ is an ideal of $L$.

If $L_1 \neq 0$ then $Z(L) \subseteq L_1 \cap L_0 = 0$, a contradiction. Hence $L_1 = 0$ and $R_x$ is nilpotent. But then $L =M + Fx$ is nilpotent and hence abelian, and the result follows.
\end{proof}
\medskip

Next we aim to show the relationship between ideals of $L$ and the decomposition given in Corollary \ref{c:decomp}. First we need the following lemma.

\begin{lemma}\label{l:ideal}
Let $L$ be a solvable Leibniz $A$-algebra of derived length $\leq n+1$, and suppose that $L = B \dot{+} C$ where $B = L^{(n)}$ and $C$ is a subalgebra of $L$. If $D$ is an ideal of $L$ then $D = (B \cap D) \dot{+} (C \cap D)$.
\end{lemma}
\begin{proof} Let $L$ be a counter-example for which dim$L$ + dim$D$ is minimal. Suppose first that $D^2 \neq 0$. Then $D^2 = (B \cap D^2) \dot{+} (C \cap D^2)$ by the minimality of $L$. Moreover, since
$$ L/D^2 = (B + D^2)/D^2 \hspace{.2cm} \dot{+} \hspace{.2cm} (C + D^2)/D^2 $$
we have
$$ D/D^2 = (B \cap D + D^2)/D^2 \hspace{.2cm} \dot{+} \hspace{.2cm} (C \cap D + D^2)/D^2 $$
whence 
$$D = B \cap D + C \cap D + D^2 = B \cap D \dot{+} C \cap D.$$
We therefore have that $D^2 = 0$. Similarly, by considering $L/B \cap D$, we have that $B \cap D = 0$. 
\par
Put $E = C^{(n-1)}$. Then $(D + B)/B$ and $(E + B)/B$ are abelian ideals of the Leibniz $A$-algebra $L/B$, and so 
$$ \left[ \frac{D + B}{B}, \frac{E + B}{B} \right]+ \left[ \frac{E + B}{B}, \frac{D + B}{B} \right] = \frac{B}{B}, $$
by Lemma \ref{l:lemm1} (ii), whence
$$ [D, E]+[E,D] \subseteq [D + B, E + B]+[E+B,D+B] \subseteq B \hbox{ and }$$ $$ [D,E]+[E,D] \subseteq B \cap D = 0;$$ that is, $D \subseteq Z_L(E)$. But $Z_L(E) = Z_B(E) + Z_C(E)$. For, suppose that $x = b + c \in Z_L(E)$, where $b \in B$, $c \in C$. Then $0 = [x,E] = [b,E] + [c,E]$, so $[b,E] = - [c,E] \in B \cap C = 0$. Similarly, $[E,b]=-[E,c]=0$, so that $Z_L(E) \subseteq Z_B(E) + Z_C(E)$. But the reverse inclusion is clear, so equality follows.
\par
Now $L^{(n-1)} \subseteq B + E \subseteq L^{(n-1)}$, so $B = L^{(n)} = (B + E)^2 = [B,E]+[E,B]$. But
$$ [E,B]\subseteq [[E,L^{(n-1)}],L^{(n-1)}]=[[E,B+E],B+E]\subseteq [B,B+E]=[B,E],$$ so $B=[B,E]$.
Let $L^{(n-1)} = L_0 \dot{+} L_1$ be the Fitting decomposition of $L^{(n-1)}$ relative to $R_E$. Then $B \subseteq L_1$ so that $Z_B(E) \subseteq L_0 \cap L_1 = 0$, whence $D\subseteq Z_L(E) = Z_C(E) \subseteq C$ and the result follows.  
\end{proof}

\begin{theor}\label{t:nz}
Let $L$ be a solvable Leibniz $A$-algebra of derived length $n+1$ with nilradical $N$, and let $K$ be an ideal of $L$ and $A$ a minimal ideal of $L$. Then, with the same notation as Corollary \ref{c:decomp}, 
\begin{itemize}
\item[(i)] $K = (K \cap A_n) \dot{+} (K \cap A_{n-1}) \dot{+} \ldots \dot{+} (K \cap A_0)$;
\item[(ii)] $N = A_n \oplus (N \cap A_{n-1}) \oplus \ldots \oplus (N \cap A_0)$;
\item[(iii)] $Z(L^{(i)}) = N \cap A_i$ for each $0 \leq i \leq n$; and
\item[(iv)] $A \subseteq N \cap A_i$ for some $0 \leq i \leq n$. 
\end{itemize} 
\end{theor}
\begin{proof} (i) We have that $L = A_n \dot{+} B_n$ where $A_n = L^{(n)}$ from the proof of Corollary \ref{c:decomp}. It follows from Lemma \ref{l:ideal} that $K = (K \cap A_n) + (K \cap B_n)$. But now $K \cap B_n$ is an ideal of $B_n$ and $B_n = A_{n-1} \dot{+} B_{n-1}$. Applying Lemma \ref{l:ideal} again gives $K \cap B_n = (K \cap A_{n-1}) \dot{+} (K \cap B_{n-1})$. Continuing in this way gives the required result.
\par
(ii) This is clear from (i), since $A_n = L^{(n)} = N \cap A_n$.
\par
(iii) We have that $L^{(i)} = L^{(i+1)} \dot{+} A_i$ from Corollary \ref{c:decomp}, and also that $Z(L^{(i)}) \cap L^{(i+1)} = 0$ from Theorem \ref{t:int}. Thus, using Lemma \ref{l:ideal},
$$ Z(L^{(i)}) = (Z(L^{(i)}) \cap L^{(i+1)}) + (Z(L^{(i)}) \cap A_i) = Z(L^{(i)}) \cap A_i \subseteq N \cap A_i. $$
It remains to show that $N \cap A_i \subseteq Z(L^{(i)})$; that is, $[N \cap A_i,L^{(i)}]+[L^{(i)},N\cap A_i] = 0$. 
We use induction on the derived length of $L$. If $L$ has derived length one the result is clear. So suppose it holds for Leibniz algebras of derived length $\leq k$, and let $L$ have derived length $k+1$. Then $B = A_{k-1} + \dots + A_0$ is a solvable Leibniz $A$-algebra of derived length $k$, and, if $N$ is the nilradical of $L$, then $N \cap A_i$ is inside the nilradical of $B$ for each $0 \leq i \leq k-1$, so $[N \cap A_i, B^{(i)}]+[B^{(i)},N\cap A_i] = 0$ for $0 \leq i \leq k-1$, by the inductive hypothesis. But $[N \cap A_i, A_k] = [N \cap A_i,L^{(k)}] \subseteq [N,N] = 0$, for $0 \leq i \leq k$, whence $[N \cap A_i,L^{(i)}] = [N \cap A_i,A_k + B^{(i)}] = 0$ for $0 \leq i \leq k$. Similarly, $[L^{(i)},N\cap A_i]=0$.
\par
(iv) We have $A \subseteq L^{(i)}$, $A \not \subseteq L^{(i+1)}$ for some $0 \leq i \leq n$. Now $[L^{(i)}, A] \subseteq [L^{(i)}, L^{(i)}] = L^{(i+1)}$, so $[L^{(i)}, A] \neq A$. It follows that $[L^{(i)}, A] = 0$. Similarly, $[A,L^{(i)}]=0$, whence $A \subseteq Z(L^{(i)}) = N \cap A_i$, by (ii).
\end{proof}
\medskip

The final result in this section shows when two ideals of a Leibniz $A$-algebra centralise each other.

\begin{propo}\label{p:cent}
Let $L$ be a Leibniz $A$-algebra and let $B, D$ be ideals of $L$. Then $B \subseteq Z_L(D)$ if and only if $B \cap D \subseteq Z(B) \cap Z(D)$.
\end{propo}
\begin{proof} Suppose first that $B \subseteq Z_L(D)$. Then $[B \cap D,D]+[D,B\cap D] = 0 = [B \cap D,B]+[B,B\cap D]$, whence $B \cap D \subseteq Z(B) \cap Z(D)$.
\par
Conversely, suppose that $B \cap D \subseteq Z(B) \cap Z(D)$. Then $[B,D]+[D,B] \subseteq B \cap D \subseteq Z(B + D)$ which yields that $[B,D]+[D,B] \subseteq (B + D)^2 \cap Z(B + D) = 0$, by Theorem \ref{t:int}. Hence $B \subseteq Z_L(D)$.
\end{proof}

\section{Completely solvable Leibniz $A$-algebras}
\medskip
A Leibniz algebra $L$ is called {\it completely solvable} if $L^2$ is nilpotent. Over a field of characteristic zero every solvable Leibniz algebra is completely solvable. Clearly completely solvable Leibniz $A$-algebras are metabelian so we would expect stronger results to hold for this class of algebras. First the decomposition theorem takes on a simpler form.

\begin{theor}\label{t:ss}
Let $L$ be a completely solvable Leibniz $A$-algebra with nilradical $N$. Then $L = L^2 \dot{+} B$, where $L^2$ is abelian and $B$ is an abelian subalgebra of $L$, and $N = L^2 \oplus Z(L)$.
\end{theor}
\medskip
{\it Proof.} We have that $L = L^2 \dot{+} B$, where $B$ is an abelian subalgebra of $L$, by Theorem \ref{t:split}. Also, $L^2$ is nilpotent and so abelian. Moreover, $N = L^2 + N \cap B$ and $N \cap B = Z(L)$, by Theorem \ref{t:nz}.
\medskip

Next we see that the minimal ideals are easy to locate.

\begin{theor}\label{t:min}
Let $L = L^2 \dot{+} B$ be a completely solvable Leibniz $A$-algebra and let $A$ be a minimal ideal of $L$. Then 
\begin{itemize}
\item[(i)] $A \subseteq L^2$ or $A \subseteq B$;
\item[(ii)] $A \subseteq B$ if and only if $A \subseteq Z(L)$ (in which case dim $A = 1$); and
\item[(iii)] $A \subseteq L^2$ if and only if $[A,L] = A$.
\end{itemize} 
\end{theor}
\begin{proof} (i) and (ii) follow from Theorem \ref{t:nz} (iii) and (iv).
\par
(iii) Suppose that $A \subseteq L^2$. Then $[A,L]+[L,A] \neq 0$ from (ii), so $[A,L]+[L,A] = A$. But $[L,A]=0$ or $[x,a]=-[a,x]$ for all $x\in L, a\in A$, by \cite[Lemma 1.9]{barnes}. Hence $[A,L]=A]$.
\par

The converse is clear.
\end{proof}

\begin{coro}\label{c:phi}
Let $L$ be a completely solvable Leibniz $A$-algebra. Then $L$ is $\phi$-free if and only if $L^2 \subseteq $ Asoc$L$.
\end{coro}
\begin{proof} Suppose first that $L$ is $\phi$-free. Then $L^2 \subseteq N = $ Asoc$L$, by \cite[Theorem 2.4]{stit}. 
\par
So suppose now that $L^2 \subseteq $ Asoc$L$. Then $L$ splits over Asoc$L$ by Theorem \ref{t:split}. But now $L$ is $\phi$-free by \cite[Proposition 3.1]{stit}.
\end{proof}
\medskip

Finally we can identify the maximal nilpotent subalgebras of $L$. First we need the following lemma.

\begin{lemma}\label{l:maxn}
Let $L$ be a metabelian Leibniz algebra, and let $U$ be a maximal nilpotent subalgebra of $L$. Then $U \cap L^2$ is an abelian ideal of $L$ and $L^2 = (U \cap L^2) \oplus K$ where $K$ is an ideal of $L$ and $[K,U] = K$.
\end{lemma}
\begin{proof} Let $L = L_0 \dot{+} L_1$ be the Fitting decomposition of $L$ relative to $R_U$. Then $L_1 = \cap_{i=1}^{\infty} L({\rm ad}\,U)^i \subseteq L^2$, and so $L^2 = (L_0 \cap L^2)\dot{+} L_1$. Now 
$$[L,L_0 \cap L^2] = [L_0 + L_1,L_0 \cap L^2] \subseteq (L_0 \cap L^2) + L^{(2)} = L_0 \cap L^2.$$ Similarly, $[L_0\cap L^2,L]\subseteq L_0\cap L^2$ 
so $L_0 \cap L^2$ is an ideal of $L$. Also, $U^2\subseteq L_0\cap L^2$ and an induction argument similar to that in Lemma \ref{l:leftideal} shows that $L_U^k(L_0\cap L^2)\subseteq R_U^{k-1}(L_0\cap L^2)$ for $k\geq 1$. It follows that $U + (L_0 \cap L^2)$ is a nilpotent subalgebra of $L$, and so $L_0 \cap L^2 \subseteq U \cap L^2$. The reverse inclusion is clear. 
\par

Next, $[L^2,L_1]\subseteq L^{(2)}=0$, so $[L^2,U]=[L_1,U]=L_1$. But now, 
\[[L_0,L_1]\subseteq [L_0,[L^2,U]]\subseteq [[L_0,L^2],U]+[[L_0,U],L^2]\subseteq [L^2,U]=L_1,
\] so $L_1$ is an ideal of $L$. Hence we can put $K=L^2$.
\end{proof} 

\begin{theor}\label{t:maxn}
Let $L$ be a completely solvable Leibniz $A$-algebra, and let $U$ be a maximal nilpotent subalgebra of $L$. Then $U = (U \cap L^2) \oplus (U \cap C)$ where $C$ is a Cartan subalgebra of $L$.
\end{theor}
\begin{proof} Put $U = (U \cap L^2) \oplus D$, so $D$ is an abelian subalgebra of $L$. Let $L = L_0 \dot{+} L_1$ be the Fitting decomposition of $L$ relative to $R_D$. As in Lemma \ref{l:maxn}, $L_1$ is an abelian right ideal of $L$. 
\par
Now put $L^2 = (U \cap L^2) \oplus K$ as given by Lemma \ref{l:maxn}. Then 
$$K = [K,U] = [K,D] \hbox{ so } K \subseteq L_1 \hbox{ and } U \cap L^2 \subseteq L_0 \cap L^2.$$ 
Hence 
$$L_0^2\subseteq L_0 \cap L^2 = (U \cap L^2) + (L_0 \cap K) = U \cap L^2,$$ since $L_0\cap K\subseteq L_0\cap L_1=0$.
\par 
Next put $L_0 = L_0^2 \dot{+} E$ where $E$ is an abelian subalgebra of $L_0$. Then 
$$U = L_0 \cap U = L_0^2 \oplus (E \cap U) = (U \cap L^2) \oplus (E \cap U). \hspace{1cm} (*)$$
\par
Finally put $E = (E \cap L^2) \oplus C$ where $E \cap U \subseteq C$. Then 
$$L = L_1 + L_0 = L^2 + L_0 = L^2 + E = L^2 \dot{+} C $$ 
so $C$ is a Cartan subalgebra of $L$, by Theorem \ref{t:split}. Moreover, $E \cap U \subseteq C \cap U$, so (*) implies that 
$$C \cap U = (E \cap U) \oplus (C \cap U \cap L^2) = E \cap U,$$ 
since $C \cap L^2 = 0$. But now (*) becomes $U = (U \cap L^2) \oplus (U \cap C)$ where $C$ is a Cartan subalgebra of $L$, as claimed.   
\end{proof}

\section{Monolithic solvable Leibniz $A$-algebras}
Monolithic Lie algebras play a part in the application of Lie $A$-algebras to the study of residually finite varieties, so it seems worthwhile to investigate whether the extra properties they have are inherited by their Leibniz counterparts

\begin{theor}\label{t:mon}
Let $L$ be a monolithic solvable Leibniz $A$-algebra of derived length $n+1$ with monolith $W$. Then, with the same notation as Corollary \ref{c:decomp}, 
\begin{itemize}
\item[(i)] $W$ is abelian;
\item[(ii)] $Z(L) = 0$ and either $[L,W] = W$ or $[W,L]=W$; 
\item[(iii)] $N = A_n = L^{(n)}$; 
\item[(iv)] $N = Z_L(W)$; and
\item[(v)] $L$ is $\phi$-free if and only if $W = N$.
\end{itemize} 
\end{theor}
\begin{proof} (i) Clearly $W \subseteq L^{(n)}$, which is abelian.
\par
(ii) If $Z(L) \neq 0$ then $W \subseteq Z(L) \cap L^2 = 0$, by Theorem \ref{t:int}, a contradiction. Hence $Z(L) = 0$. It follows from this that $[L,W]+[W,L] \neq 0$. But $[L,W]$ is an ideal of $L$, so either $[L,W] = W$ or $[L,W]=0$, in which case $[W,L]=0$..
\par
(iii) We have $N = A_n \oplus N \cap A_{n-1} \oplus \ldots \oplus N \cap A_0$ by Theorem \ref{t:nz}(i). Moreover, $N \cap A_i$ is an ideal of $L$ for each $0 \leq i \leq n-1$, by Theorem \ref{t:nz}(iii). But if $N \cap A_i \neq 0$ then $W \subseteq A_n \cap N \cap A_i = 0$ if $i \neq n$. This contradiction yields the result.
\par
(iv) We have that $L = N \dot{+} B$ for some subalgebra $B$ of $L$, by Theorem \ref{t:split} and (iii). Put $C = Z_L(W)$ and note that $N \subseteq C$. Suppose that $N \neq C$. Then $C = N \dot{+} B \cap C$. Choose $A/N$ to be a minimal ideal of $L/N$, so that $A^2\subseteq N$. Pick $x\in A\setminus N$ and let $L = L_0 \dot{+} L_1$ be the Fitting decomposition of $L$ relative to $R_x$. Then 
$$L_1 = \bigcap_{i=1}^{\infty} R_x^i(L) \subseteq [[L,A],A] \subseteq [A,A] \subseteq N,$$ 
which is abelian. Hence $N=L_1\dot{+}N\cap L_0$. Now $N\cap L_0$ is an ideal of $L$, since $[L_1,N\cap L_0]+[N\cap L_0,L_1]\subseteq N^2=0$ and it is clearly invariant under $L_0$. Moreover, $Fx+N\cap L_0$ is a nilpotent subalgebra of $L$, since $x^2\in$ Leib$(L)\subseteq N$, $x^2\in L_0$ and using Lemma \ref{l:leftideal}. Hence it is abelian, and so $[N\cap L_0,x]=0$ and $[N,x]=[L_1,x]=L_1$. It follows that $L_1=R_x^k(N)$ for all $k\geq 1$. But now, a straightforward induction proof shows that $[L_0,L_1]\subseteq [L_0,R_x^k(N)]\subseteq L_1+[R_x^k(L_0),N]$ for all $k\geq 1$. Since $R_x^k(L_0)=0$ for some $k$ this yields that $[L_0,L_1]\subseteq L_1$. Thus $L_1$ is an abelian ideal of $L$, and so $L_1=0$, as, otherwise, $W\subseteq L_1\cap L_0=0$.This yields that $Fx+N$ is nilpotent and thus abelian, whence $A \subseteq Z_L(N) \subseteq N$, by Lemma \ref{l:nilrad}. This contradiction implies that $N = C$.
\par
(v) Clearly $W =$ Asoc$L$. Suppose first that $L$ is $\phi$-free. Then $W =$ Asoc$L = N$, by \cite[Theorem 7.4]{frat}. So suppose now that Asoc$L = W = N$. Then $L$ splits over Asoc$L$ by Theorem \ref{t:split} and (iii). But now $L$ is $\phi$-free by \cite[Theorem 7.3]{frat}.
\end{proof}
\medskip

It is shown in \cite{tow} that monolithic solvable Lie $A$-algebras are not necessarily metabelian. However, when a Leibniz $A$-algebra is strongly solvable the situation is more straightforward.

\begin{theor}\label{t:ssmon}
Let $L$ be a monolithic strongly solvable Leibniz $A$-algebra. Then the maximal nilpotent subalgebras of $L$ are $L^2$ and the Cartan subalgebras of $L$ (that is, the subalgebras that are complementary to $L^2$.) 
\end{theor}
\begin{proof} Let $U$ be a maximal nilpotent subalgebra of $L$ and let $W$ be the monolith of $L$. Then $L^2 = (U \cap L^2) \oplus K$ where $U \cap L^2, K$ are ideals of $L$ and $[U,K] = K$, by Lemma \ref{l:maxn}. Either $W \subseteq U \cap L^2$ and $K = 0$ or else $W \subseteq K$ and $U \cap L^2 = 0$.
\par
In the former case $N = L^2 \subseteq U$, by Theorem \ref{t:mon}. But then $U \subseteq Z_L(N) \subseteq N$, by Lemma \ref{l:nilrad}, so $U = L^2$. In the latter case $U$ is a Cartan subalgebra of $L$, by Theorem \ref{t:maxn}. 
\end{proof}
\medskip

Finally we give necessary and sufficient conditions for a monolithic algebra to be a strongly solvable Leibniz $A$-algebra. 

\begin{lemma}\label{l:aa}
Let $L = L^2 \dot{+} B$ be a metabelian Leibniz algebra, where $B$ is a subalgebra of $L$, and suppose that $[L^2,b] = L^2$ for all $b \in B$. Then $L$ is a strongly solvable $A$-algebra.
\end{lemma}
\begin{proof} Let $U$ be a maximal nilpotent subalgebra of $L$. We have $L^2 = (U \cap L^2) \oplus K$ where $K$ is an ideal of $L$ and $[U,K] = K$, by Lemma \ref{l:maxn}. Let $u = x + b \in U$, where $x \in L^2$, $b \in B$. Then $L^2 = [L^2,b] = [L^2,u]$, so $L^2 = R_u^i(L^2)$ for all $i \geq 1$. It follows that $L^2 = K$ from which $U^2 \subseteq U \cap L^2 = 0$ and $L$ is an $A$-algebra.
\end{proof}

\begin{theor}\label{l:mona}
Let $L$ be a monolithic Leibniz algebra. Then $L$ is a strongly solvable $A$-algebra if and only if $L = L^2 \dot{+} B$ is metabelian, where $B$ is a subalgebra of $L$ and $[L^2,b] = L^2$ for all $b \in B$ (or, equivalently, $R_b$ acts invertibly on $L^2$).
\end{theor}
\begin{proof} Suppose first that $L$ is a strongly solvable $A$-algebra. Then $L = L^2 \dot{+} B$ is metabelian, where $B$ is a subalgebra of $L$, by Theorem \ref{t:split}. Let $b \in B$ and let $L = L_0 \dot{+} L_1$ be the Fitting decomposition of $L$ relative to $R_b$. It is easy to see, as in Lemma \ref{l:maxn}, that $L^2 = (L^2 \cap L_0) \dot{+} L_1$ and $L^2 \cap L_0$ and $L_1$ are ideals of $L$, so $L^2 = L^2 \cap L_0$ or $L^2 = L_1$ as $L$ is monolithic. The former implies that $[L^2,b] = 0$. But then 
\[ [b,[b,L^2]]\subseteq [b^2,L^2]+[[b,L^2],b]\subseteq [L^2,b]=0,
\] so $L^2+Fb$ is a nilpotent subalgebra of $L$ and hence is abelian. This yields that $L^2$ and $Fb$ are ideals of $L$, which is impossible. It follows that $L^2 = L_1$, whence $[L^2,b] = L^2$. If $\theta = R_b|_{L^2}$ then $L^2 = {\rm Ker}\,\theta \dot{+} {\rm Im}\,\theta$, so ${\rm Ker}\,\theta = \{0\}$ and $\theta$ is invertible.
\par
The converse follows from Lemma \ref{l:aa}.
\end{proof}

\section{Cyclic Leibniz algebras}
{\it Cyclic} Leibniz algebras, $L$, are generated by a single element. In this case $L$ has a basis $a,a^2, \ldots, a^n (n > 1)$ and product $[a^n,a]=\alpha_2a^2+ \ldots + \alpha_na^n$.  Let $T$ be the matrix for $R_a$ with respect to the above basis. Then $T$ is the
companion matrix for $p(x) =  x^n - \alpha_n x^{n-1} - \ldots - \alpha_2 x= p_1(x)^{n_1} \ldots p_r(x)^{n_r}$, where the
$p_j$ are the distinct irreducible factors of $p(x)$. Then we have the following result.

\begin{theor} $L$ is a cyclic Leibniz $A$-algebra if and only if $\alpha_2 \neq 0$, and then $L=L^2\dot{+} F(a^n-\alpha_n a^{n-1} - \dots - \alpha_2 a)$ and we can take $p_1(x)^{n_1}=x$. \end{theor}
\begin{proof} If $\alpha_2=0$ we have that $L$ is nilpotent but not abelian, so $L$ is not an $A$-algebra. If $\alpha_2\neq 0$, it is easy to check that $Fb=F(a^n-\alpha_n a^{n-1} - \dots - \alpha_2 a)$ is a subalgebra of $L$ which complements $L^2$, and $[L^2,b]=L^2$. It follows from Lemma \ref{l:aa} that $L$ is an $A$-algebra. Moreover, $p(x)$ is divisible by $x$ only once.
\end{proof}

\begin{theor}\label{t:cyclicmono} The cyclic Leibniz $A$-algebra $L$ is monolithic if and only if $p(x)$ has exactly two irreducible factors (one of which is $x$).
\end{theor}
\begin{proof} This follows easily from \cite[Corollary 4.5]{bat}.
\end{proof}

\begin{coro}The cyclic Leibniz $A$-algebra $L$ is monolithic and $\phi$-free if and only if $p(x)=xp_2(x)$ \end{coro}
\begin{proof} Theorem \ref{t:cyclicmono} and \cite[Corollary 4.2]{bat}.
\end{proof}

\begin{coro}If the underlying field is algebraically closed, then the  cyclic Leibniz $A$-algebra $L$ is monolithic and $\phi$-free if and only if it is two dimensional with $[a^2,a]=a^2$.\end{coro}
\begin{proof} Clearly $p(x)$ is quadratic, so $L$ is two dimensional, and replacing $a$ by $(1/\sqrt{\alpha_2})a$ gives the claimed multiplication.
\end{proof}

\section{Solvable Leibniz $A$-algebras over an algebraically closed field}
The following result was proved for Lie algebras by Drensky in \cite{dren}. 

\begin{theor}\label{t:dren} Let $L$ be a solvable Leibniz $A$-algebra over an algebraically closed field $F$. Then the derived length of $L$ is at most $3$.
\end{theor}
\medskip
\begin{proof} First note that we can assume that the ground field is of characteristic $p > 0$, since otherwise $L$ is strongly solvable and so of derived length at most $2$. Suppose that $L$ is a minimal counter-example, so the derived length of $L$ is four. 
\par
Let $A$ be a minimal ideal of $L$ contained in Leib$(L)$, and put $N=L^{(2)}$. We have that $L^{(3)}=A$.  Put $\bar{L} = L/$Leib$(L)$ and for each $x \in L$ write $\bar{x} = x +$ Leib$(L)$. Then $A$ is an irreducible right $\bar{L}$-module, and hence an irreducible right $U$-module, where $U$ is the universal enveloping algebra of $\bar{L}$. Let $\phi$ be the corresponding representation of $U$ and let $\bar{x} \in \bar{L}$, $n \in N$. Then $[[\bar{x},\bar{n}],\bar{n}] = \bar{0}$, whence $[\bar{x},\bar{n}^p] = 0$ and so $\bar{n}^p \in Z = Z(U)$. 
\par
Let $n_1, n_2 \in N$. Then $\bar{n}_1^p, \bar{n}_2^p \in Z$, so $\alpha_1 \bar{n}_1^p + \alpha_2 \bar{n}_2^p \in \hbox{ker}(\phi)$, for some $\alpha_1, \alpha_2 \in F$, since dim $\phi(Z) \leq 1$, by Schur's Lemma. Since $F$ is algebraically closed, there are $\beta_1, \beta_2 \in F$ such that $\alpha_1 = \beta_1^p, \alpha_2 = \beta_2^p$, so $(\beta_1 \bar{n}_1 + \beta_2 \bar{n}_2)^p = \beta_1^p \bar{n}_1^p + \beta_2^p \bar{n}_2^p \in \hbox{ker}(\phi)$, since $[\bar{n}_1,\bar{n}_2] = \bar{0}$. It follows from this together with Lemma \ref{l:leftideal} that $A + F(\beta_1 n_1 + \beta_2 n_2)$ is a nilpotent subalgebra of $L$ and hence abelian. Thus $\beta_1 \bar{n}_1 + \beta_2 \bar{n}_2 \in \hbox{ker}(\phi)$ and so dim $\phi(\bar{N}) \leq 1$. Hence $Z_N(A)$ has codimension at most $1$ in $N$.

Then $\dim N/Z_N(A)) \leq 1$.
Suppose that $\dim N/Z_N(A)) = 1$. Put $S = L/Z_N(A)$. Then dim$(S^{(2)}) = 1$. It follows that $S/Z_L(S^{(2)}) \subseteq R_S(S^{(2)})$ and so has dimension at most one, giving $[S^{(1)},S^{(2)}]+[S^{(2)},S^{(1)}] = 0$. But now $S^{(1)}$ is nilpotent but not abelian. As $S$ must be an $A$-algebra, this is a contradiction. We therefore have that dim $(L^{(2)}/Z_{L^{(2)}}(A)) = 0$, whence $[A,L^{(2)}] = 0$.
\par
Now we can include $L^{(3)}$ in a chief series for $L$. So let $0 = A_0 \subset A_1 \subset \ldots \subset A_r = L^{(3)}$ be a chain of ideals of $L$ each maximal in the next. By the above we have $[A_i,L^{(2)}] \subseteq A_{i-1}$ for each $1 \leq i \leq r$. It follows that $L^{(2)}$ is a nilpotent subalgebra of $L$ and hence abelian. We infer that $L^{(3)} = 0$, a contradiction. The result follows.
\end{proof}

\end{document}